\documentclass[10pt,amsfonts, epsfig]{amsart}

\usepackage{amsmath, amscd, amssymb}
\usepackage[frame,cmtip,arrow,matrix,line,graph,curve]{xy}
\usepackage{graphpap, color}
\usepackage[mathscr]{eucal}
\usepackage{mathrsfs}

\usepackage{pstricks}
\usepackage{color}
\usepackage{cancel}
\def\<{\langle}
\def\>{\rangle}
\numberwithin{equation}{section}
\def\cM{{\cal M}}
\def\fM{{\mathfrak M}}
\def\Po{{\mathbb P^1}}

\newcommand{\CC}{\mathbb{C}}

\newcommand{\RR}{\mathbb{R}}
\newcommand{\ZZ}{\mathbb{Z}}
\def\Z{\mathcal{Z}}
\def\I{\mathcal{I}}

\def\sO{{\mathscr O}}

\newcommand{\cal}{\mathcal}

\def\cM{{\cal M}}

\def\cO{{\cal O}}

\def\sO{{\mathscr O}}

\def\fM{\mathfrak{M}}

\def\mapright#1{\,\smash{\mathop{\lra}\limits^{#1}}\,}

\def\lra{\longrightarrow}

\def\begeq{\begin{equation}}
\def\endeq{\end{equation}}
\def\and{\quad{\rm and}\quad}

\usepackage{amsmath, amscd, amssymb}
\usepackage[frame,cmtip,arrow,matrix,line,graph,curve]{xy}
\usepackage{graphpap, color}
\usepackage[mathscr]{eucal}

\numberwithin{equation}{section}

\usepackage{amsmath, amscd, amssymb}
\usepackage[frame,cmtip,arrow,matrix,line,graph,curve]{xy}
\usepackage{graphpap, color}
\usepackage[mathscr]{eucal}

\numberwithin{equation}{section}

\def\Po{{\mathbb P^1}}
\def\and{\quad\text{and}\quad}
\def\mapright#1{\,\smash{\mathop{\lra}\limits^{#1}}\,}

\newtheorem{prop}{Proposition}[section]

\newtheorem{lemm}[prop]{Lemma}
\newtheorem{coro}[prop]{Corollary}

\newtheorem{defi}[prop]{Definition}

\newtheorem{defiprop}[prop]{Definition-Proposition}

\def\sO{{\mathscr O}}

\def\Ob{\cO b}

\let\lab=\label

\title{On Donaldson Thomas invariants of $\Po$ scroll}
\author{Huai-Liang, Chang}

\begin{document}
\maketitle
\begin{abstract}
  Let $S$ be a smooth algebraic surface and let $L$ be a line bundle on $S$. Suppose $\sigma$ is a
  holomorphic two-form on $S$ with smooth degeneracy loci $C$.
  Consider the Donaldson-Thomas invariant (\cite{MNOP2}) of
  $X=P(L\oplus \sO_S)$ with prime field insertions.
  We show that $\sigma$ localizes the virtual fundamental class of the moduli of ideal sheaves
  $I_n(X,\beta)$ to $D=P(L|_C\oplus \sO_C)$. When $X$ is proper, insertions lie in $D$
  and $L=\sO_S(nC)$ for some $n$, one can define the \emph{localized DT-invariants}. Compared
  to the GW case by \cite{tom} and \cite{Jli}, it gives an evidence of M.N.O.P. conjectures on identifying
  the GW and the DT theories. It is shown to be deformation invariant and depends only on the topology
  of $X$, namely genus of $C$, the theta characteristic and the degree of $L$.\\
\end{abstract}

\section{\bf Introduction}

    In \cite{MNOP2}, D. Maulik, N. Nekrasov, A. Okounkov, R. Pandharipande raised several conjectures about DT invariants.
    Let $X$ be an arbitrary smooth projective three-fold. Consider the special case of moduli of stable sheaves, $I_n(X,\beta)$,
    which is the moduli space of ideal sheaves that correspond to dimension one subschemes  $Z$ in $X$ where $[Z]=\beta \in H_2(X)$ and
    $\chi(\sO_Z)=n$. The DT-invariants were defined by intersecting descendant classes $\gamma_i$ with the virtual
    cycle $[{I}_n(X,\beta)]^{vir}$. Associated to this collection of invariants, one can form a generating function:

\begin{eqnarray*}
Z_{DT}(X;q| \prod^r_{i=1} \tau_{k_i}(\gamma_i))_{\beta}=
\sum_{n\in Z}\int_{[{I}_n(X,\beta)]^{vir}}\prod^r_{i=1}\tau_{k_i}(\gamma_i) q^n,\\
\end{eqnarray*}
and its reduced version by quotient out $Z_{DT}(X;q)_0$:\\
\begin{eqnarray*}
Z^{\prime}_{DT}(X,q | \prod _{i=1}^r \tau_{k_i}(\gamma_i))_{\beta}=Z_{DT}
(X;q| \prod^r_{i=1} \tau_{k_i}(\gamma_i))_{\beta}/
Z_{DT}(X;q)_{0}.
\end{eqnarray*}

Motivated by Gauge/String duality principle, Maulik, Nekresev, Okounkov and Pandiharipande proposed three
conjectures on these invariants and their relations to the reduced GW
series $Z^{\prime}_{GW}(X,\cdot)$:\\

\vskip5pt

\noindent\textbf{Conjecture 1}. {\sl $Z_{DT}(X,q)_0=M(-q)^{\int_X
c_3(T_x\otimes K_X)}.$}\\

\vskip5pt \noindent\textbf{Conjecture 2.} {\sl The series
$Z^{\prime}_{DT}(X,q | \prod _{i=1}^r
\tau_{k_i}(\gamma_i)_{\beta})$ is rational in $q$.}\\

\vskip5pt \noindent\textbf{Conjecture 3.} {\sl For $d=\int_{\beta}
c_1(T_x)$ and $q=-e^{iu}$
\begin{eqnarray*}
(-iu)^d Z^{\prime}_{GW}(X,u| \prod_{i=1}^r
\tau_0(\gamma_i))_{\beta} =(-q)^{-d/2}Z^{\prime}_{DT}(X,q|
\prod_{i=1}^r \tau_0(\gamma_i))_{\beta}.
\end{eqnarray*}}\\

J. Li \cite{Jli1} and M. Levine with R. Pandharipande \cite{AC}
solved the first part of the MNOP conjectures. While J. Li's
approach relies on analyzing the local behavior of the moduli space,
M. Levine and R. Pandharipande proved the first conjecture using
algebraic cobordism theory. They degenerate all three-folds into
union of toric three-folds and then applying the degeneration
formula for DT-invariants \cite{Deg}.\\

 The second conjecture and the third conjecture are under attack from various directions.
One possible approach, as proposed by R.Pandharipande, is to apply algebraic cobordism theory and the
degeneration formula of GW/DT invariants. Like the case for GW theory, B. Wu and J. Li \cite{Deg} have recently
proved the degeneration formula for DT-invariants. Since a three-fold can always be degenerated to a product of
projective spaces, and for the later one can compute their GW and DT invariants by virtual localization, an
essential part to establish conjecture 3 is to investigate the DT-invariants of all three-folds that appear in
the degeneration scheme. Specifically, when we degenerate $X$ to a union of $X_1$ and $X_2$ along a divisor D,
the GW or DT-invariants of $X$ can be re-constructed by the relative invariants of the pairs $X_1/D$ and $X_2/D$,
according to the degeneration formula. Using the standard degeneration, the relative invariants of any pair
$X_i/D$ is determined by the absolute invariants of $X_i$ and the relative invariants of a $\Po$ scroll over $D$.
Since one can use induction to further degenerate individual $X_i$, the problem essentially reduces to
determining the relative invariants of a $\Po$-bundle over a surface relative to a divisor that is a section
of this $\Po$-bundle.\\

 This work aims at determining the structure of the DT-invariants of a $\Po$ scroll over a surface of general type. We derive a vanishing theorem and a localization principle for DT-invariants of $\Po$ scroll. In view of the second conjecture, our result is parallel to the theta-localization for Gromov Witten theory on $p_g>0$ surfaces given by \cite{tom} and \cite{Jli}.\\

 The author thanks Jun Li for helpful conversations and his explanation of the two-form localization. The author also thanks Y-H. Kiem for discussion about \cite{Jli} and the derived approach. The author specially appreciates discussion with Brian Conrad on the use of Grothendieck duality in lemma (\ref{can}).

\section{\bf The obstruction sheaf and a vanishing criterion}
 \,\,\,\, We give a reinterpretation of the obstruction
 sheaves of Donaldosn-Thomas moduli spaces. It
 behaves better in Serre duality for all three-folds than just Calabi Yau three-folds.
 The second part of this section is the cosection lemma of J. Li and Y.H. Kiem (\cite{Jli}). The lemma
  plays a crucial role in our proof of the vanishing theorem.
\subsection{Obstruction sheaf in Donaldosn Thomas theory}
 \,\,\,\,Let $X$ is an arbitrary projective three-fold. Let
 $\eta: \Z\rightarrow I_n(X,\beta)$ be the universal family
 of subschemes in $X$ where $\Z$ is a subscheme of $Y:=I_n(X,\beta)\times X$
 with ideal sheaf  $\mathcal{I_Z}$.   Denote the obstruction sheaf on $I_n(X,\beta)$
 by $Ob$ and the projection $I_n(X,\beta)\times X \rightarrow I_n(X,\beta)$ by $\pi$.
 Recall the obstruction sheaf is $\Ob=Ext^{2}_{\pi}(I, I)_0$. We will show
  $Ob\cong Ext^3_{\pi}(O_{X\times S}/\mathcal{I_Z},\mathcal{I_Z})$.\\
 
  From the exact sequence
\begin{eqnarray*}
0\longrightarrow \mathcal{I_Z} \longrightarrow \sO_{Y}\longrightarrow \mathcal{O_Z}
\longrightarrow 0,
\end{eqnarray*}
 there is a long exact sequence of relative extension sheaves\\
\begin{eqnarray*}
& &\longrightarrow Ext^2_{\pi}(\sO_{Y},\mathcal{I_Z})\stackrel{f_2}\longrightarrow Ext^2_{\pi}(\mathcal{I_Z},\mathcal{I_Z})\stackrel{g}\longrightarrow Ext^3_{\pi}(\mathcal{O_Z},\mathcal{I_Z})\\
& &\longrightarrow Ext^3_{\pi}(\sO_{Y},\mathcal{I_Z})\stackrel{f_3}\longrightarrow
Ext^3_{\pi}(\mathcal{I_Z},\mathcal{I_Z})\longrightarrow Ext_{\pi}^4(\mathcal{O_Z},\mathcal{I_Z})=0.
\end{eqnarray*}\\
 The last equality follows from $dimX=3$ and base change theorem. \,\,From the inclusion $\mathcal{I_Z} \subset \sO_{Y}$ there are maps
 \begin{eqnarray*}
\rho_i:Ext^i_{\pi}(\sO_{Y},\mathcal{I_Z})\longrightarrow Ext^i_{\pi}(\sO_{Y},\sO_{Y}), \,\,\,\,\,\, i=2 , 3.
\end{eqnarray*}
 Let ${\iota_i}$ be the canonical map $Ext^i_{\pi}(\sO_{Y},\sO_{Y})\rightarrow Ext^i_{\pi}(\mathcal{I_Z},\mathcal{I_Z})$, then $\iota_i$ splits the trace map $Ext^i_{\pi}(\mathcal{I_Z},\mathcal{I_Z})\rightarrow Ext^i_{\pi}(\sO_{Y},\sO_{Y})$(ref \cite {Moduli}). 
\begin{lemm}
$\iota_i \circ \rho_i=f_i$
\end{lemm}
\begin{proof} Recall there are two local to global spectral sequences and a morphism between them induced from the map $\mathcal{I_Z}\longrightarrow \sO_Y$:
\begin{eqnarray*}
E_2^{p,q}=&H^p_{\pi}(\mathcal{E}xt^q(\sO_Y,\mathcal{I_Z}))&\Rightarrow Ex^{p+q}_{\pi}(\sO_Y,\mathcal{I_Z})\\
&\downarrow&     \,\,\,\,\,\,\,\,\,\,\,\,\,\,\,\,\,\,\, \downarrow\\
\overline{E_2}^{p,q}=&H^p_{\pi}(\mathcal{E}xt^q(\mathcal{I_Z},\mathcal{I_Z}))&\Rightarrow Ext^{p+q}_{\pi}(\mathcal{I_Z},\mathcal{I_Z})\\
\end{eqnarray*}
 Consider the map $E_{\infty}^{i,0}\longrightarrow Ext^i_{\pi}(\cdot,\mathcal{I_Z})$ in the above
diagram.
\begin{eqnarray*}
&E_{\infty}^{i,0}&\longrightarrow Ext^i_{\pi}(\sO_Y,\mathcal{I_Z})\\
&\downarrow& \,\,\,\,\,\,\,\,\,\,\,\,\,\,\,\,\,\,\, \downarrow\\
&\overline{E}_{\infty}^{i,0}&\longrightarrow Ext^i_{\pi}(\mathcal{I_Z},\mathcal{I_Z})
\end{eqnarray*}
 Since the first spectral sequence degenerates at $E_2$ term, there is 
$$E_{\infty}^{i,0}\cong H^i_{\pi}(\mathcal{E}xt^0(\sO_Y,\mathcal{I_Z}))=H^i_{\pi}(\mathcal{I_Z})$$
and the map from $E_{\infty}^{i,0}\to Ext^i_{\pi}(\sO_Y,\I_Z)$ is an isomorphism. From deformation theory the natural map
\begin{eqnarray*}
 H^2_{\pi}(\sO_Y) \longrightarrow \overline{E}_2^{i,0}=H^i_{\pi}(\mathcal{E}xt^0(\mathcal{I_Z},\mathcal{I_Z}))
\end{eqnarray*}
  lifts to a map $H^2_{\pi}(\sO_Y) \longrightarrow \overline{E}^{i,0}_{\infty}$
and its composition with the map  $\overline{E}^{i,0}_{\infty}\longrightarrow
Ext^{i}_{\pi}(\mathcal{I_Z},\mathcal{I_Z})$ is the same as the $\iota_i$. In our case
\begin{eqnarray*}
H^2_{\pi}(\sO_Y) \longrightarrow \overline{E}_2^{i,0}=H^i_{\pi}(\mathcal{E}xt^0(\mathcal{I_Z},\mathcal{I_Z}))
\end{eqnarray*}
is indeed an isomorphism since $\,\,\mathcal{E}xt^0_{\pi}(\mathcal{I_Z},\mathcal{I_Z})=\sO_Y\,$. So one
concludes
 \begin{eqnarray*}
\overline{E}^{i,0}_{\infty}=\overline{E}^{i,0}_2=H^i_{\pi}(\mathcal{E}xt^0(\mathcal{I_Z},\mathcal{I_Z})),
\end{eqnarray*}
and the following diagram commute:
\begin{eqnarray*}
E_{\infty}^{i,0}=&H^i_{\pi}(\mathcal{I_Z})&= Ext^i_{\pi}(\sO_Y,\mathcal{I_Z})\\
&\downarrow\rho_i& \,\,\,\,\,\,\,\,\,\,\,\,\,\,\,\,\,\,\, \downarrow f_i\\
\overline{E}_{\infty}^{i,0}=&H^i_{\pi}(\mathcal{O_Z})&\mapright{\iota_i} Ext^i_{\pi}(\mathcal{I_Z},\mathcal{I_Z}).
\end{eqnarray*}
This proves the lemma.
\end{proof}

\begin{lemm}\lab{iso}{There is a canonical isomorphism $Ob\mapright{\thicksim} Ext^3_{\pi}(O_\Z,\mathcal{I_Z})$.}
 \end{lemm}

\begin{proof}
 Note that the cokernel of $\rho_2$ lies in $Ext^2_{\pi}(\sO_{Y},\mathcal{O_Z})=H^2_{\pi}(\mathcal{O_Z})$. This group is zero because for each $t$ in $I_n(X,\beta)$ there is $H^2(\sO_{\Z_t})=0$ by the reason that the dimension of $\Z_t$ is less than two. Hence $\rho_2$ is surjective. The same argument show $H^3(\sO_{\Z_t})=0$ and hence $\rho_3$ an isomorphism. Since $\iota_3$ is injection this shows $f_3$ is an isomorphism and hence $g$ is a surjection. The kernal of $g$ is then equal to the image of $\iota_2$ by the claim and the fact that $\rho_2$ is a surjection. Since $\iota_2$ is  a lifting of the trace map there is an isomorphism:
\begin{equation*}
Ext^{2}_{\pi}(\mathcal{I_Z}, \mathcal{I_Z})_0\longrightarrow Ext^3_{\pi}(\mathcal{O_Z},\mathcal{I_Z}).
\end{equation*}
\end{proof}
 The above proof can be applied to the case that the parameter space is
 any scheme $T$ istead of $I_n(X,\beta)$ because the obstruction theory is perfect.

\subsection{Cosection lemma}
Let $M$ be a DM-stack with perfect obstruction theory and with a cosection $\sigma :\Ob_M \to \sO_M$ of its obstruction
sheaf; let $E$ be a vector bundle on $M$ whose sheaf of sections surjects onto $\Ob_M$; The cosection lemma is proved by  J. Li and Y.H, Kiem in \cite{Jli}. 
\begin{lemm} (J. Li; Y.H. Kiem)\lab{cosection}:
Let $\tilde{\sigma} : E\to \sO_M$ be the composite of $\sigma$ with the quotient homomorphism $E \to \Ob_M$. Then the (virtual) normal cone $W \in Z_{\ast}E$ lies in the (cone)
kernel $E(\tilde
{\sigma})$ of $\tilde{\sigma}$.
\end{lemm}
 \begin{coro}
If $\sigma: \Ob \longrightarrow \sO_X$ is a surjection, then $0^! [W]=0$ in $Z_{\ast}E$. 
\end{coro}

 In the next section the cosection lemma will be applied to derived the vanishing properties
 of $Z_{DT}$ for $\Po$ scrolls.

\section{\bf Vanishings}
 \,\,\,\,With the $\CC^{\ast}$ action on the $\Po$ scroll,
 the virtual cycle can be written through virtual localization formula of
\cite{VLF}:
\begin{equation*}
[\mathbf{\cM}]^{vir}=\iota _{\ast} \sum
\frac{[\mathbf{\cM}_i]^{vir}}{e(N_i^{vir})}  ,
\end{equation*}
where $\mathbf{\cM}$ is the moduli space $I_n(X,\beta)$. We use the section $\sigma$ to ``localize" the $[\mathbf{\cM}_i]^{vir}$, which is not the classical $\CC^{\ast}$ localization. To fix the notation,  denote $S=S_0$ and let $S_{\infty}$ be the cap of $X$. If there is an curve class $E$ on $S$, denote the same curve class on $S_0$ by $E_0$, and
the same class on $S_{\infty}$ by $E_{\infty}$. The second cohomology of $X$ is determined by:
\begin{eqnarray*}
&H^2(X,\ZZ)=H^2(S,\ZZ)\oplus \ZZ[F] \\
&E_0=E_{\infty} + n F , E\in H^2(S,\ZZ)\\
&n=(c_1(L).[E]) ,
\end{eqnarray*}
where $F$ is the fiber class.
In this section we prove the vanishing for the $\Po$ scroll
\begin{prop}
Let $\sigma$ be a holomorphic two-form on $S$ with smooth zero loci $C$. Let $I_n(X,\beta)_{\CC^{\ast}}$ be the fixed loci of 
$I_n(X,\beta)$ under the $\CC^{\ast}$ action. Suppose the horizontal component of
$\beta$ in $H^2(S,\ZZ)$ is not a multiple of $C$.  Then there is $[I_n(X,\beta)_{\CC^{\ast}}]^{vir}=0$ and by the virtual localiztion formula $[I_n(X,\beta)]^{vir}=0$.
\end{prop}

\subsection {The construction of an equivariant global cosection}

 Let $\eta:Z\longrightarrow \fM$ be the universal family of $\CC^\ast$-equivariant subschemes in $X$ with $\chi=n$ and
$c_2=\beta$. Then $\fM\cong {I}_n(X,\beta)^{\CC^\ast}$ is the maximal closed subscheme of $I_n(X,\beta)$ fixed by the $\CC^{\ast}$ action. Let $\I$ be the ideal sheaf of $Z$ in $X\times \fM$. Let $q:X\times \fM\to X$ and $p:X\times\fM\to \fM$ be the projections.

Note that here $\CC^{\ast}$ acts on everything.
 By the lemma (\ref{iso}) in previous section, there is an isomorphism:
 \begin{eqnarray*}
Ob_{\eta}=Ext^{2}_{p_\ast}(\I, \I)_0 \cong Ext^3_{p_\ast}(\sO_Z, \I)
 \end{eqnarray*}

between sheaves on $\fM$. By Serre duality 
$$Ob_{\eta}^{\vee}\cong Ext^3_{p_\ast}(\sO_Z, \I)^{\vee}\cong Hom_{p_\ast}(\I, \sO_Z \otimes q^{\ast}K_X).$$

From the section $\sigma$ of the canonical bundle, we construct an
equivariant  section $\xi$ of $Hom_{p_\ast}(\I, \sO_Z\otimes q^{\ast}K_X)$ that is nonzero at every geometric point in $\fM$.  Let  $X_0$ be the total space of $L$ over $S_0$ and $X_{\infty}$ the total space of $L^{\vee}$ over $S_{\infty}$.
Then $X=X_0\cup X_{\infty}$. For $i=0,\infty$ let $Z_i=Z\cap (X_i\times \fM)$, $I_i$ the ideal sheaf of $Z_i$ in $X_i\times \fM$ and $\eta_i:=\eta \mid_{Z\cap (X_i\times \fM)}$.
To exhibit a global section of $Hom_{p_\ast}(\I, \sO_Z\otimes q^{\ast}K_X)$
 we construct a section of 
 $Hom_{p_\ast}(\I_0, \sO_{Z_0}\otimes
 q^{\ast}K_{X_0})$ and a section of  $Hom(\I_\infty, \sO_{Z_\infty}\otimes q^{\ast}K_X)$ such that both are zero homomorphisms on $(X_0\cap X_{\infty}) \times \fM$.
 Fix a trivialization of $L$ over a covering $\{U_{\alpha}\}$ of $S_0$ :
 \begin{eqnarray*}
L|_{U_\alpha}\cong U_\alpha \times \CC\mapright{t_\alpha} \CC.
\end{eqnarray*}
 Denote $X_\alpha=U_\alpha \times \CC$. Then there is $K_{X_\alpha}\cong\pi^{\ast}K_{U_\alpha}\otimes \pi^{\ast} L^{\vee}| _{U_\alpha}.$ The scheme $Z\cap(X_{\alpha} \times \fM)$ is $\CC^{\ast}$-equivariant inside $X_\alpha \times \fM$. There is $\sO_{X_\alpha}\cong \sO_{U_\alpha}[t_\alpha]$ and $t_\alpha$ has weight $-1$ with respect to the $\CC^\ast$ action. The ideal sheaf $\I_0|_{Z\cap(X_{\alpha} \times \fM)}$ is a $\CC^\ast$ invariant subsheaf of $\sO_{U_{\alpha}\times\fM}[t_\alpha]$. It is of the form $ \oplus \I_n^{\alpha} t_{\alpha}^n$
where $\I_0^\alpha\subset\I_1^\alpha\subset \I_2^\alpha\cdots$ are subsheaves of $\sO_{U_\alpha \times\fM}$. By finite ness for some $k$ large enough there is $\I_k^\alpha=\I_{k+1}^\alpha=\I_{k+2}=\cdots$. We denote the section $t_\alpha=1$ of $\pi^\ast L|_{U_\alpha}\cong \L|_{U_\alpha}$ by $v_\alpha$ and its dual basis $\widehat{v}_\alpha$ of $L^{\vee}|_{U_\alpha}$. Define
\begin{eqnarray*}
\xi_\alpha: \oplus \I_n^\alpha t_\alpha ^n \lra \oplus(\sO_{U_\alpha\times \fM}/ \I_n^\alpha)\, t_\alpha ^n\otimes q^{\ast} K_{X_{\alpha}},
\end{eqnarray*}
such that $\xi_\alpha(f\,t_\alpha^0)=0,$ and for $n\geq 0$, $f\in \I_{n+1}^{\alpha}$
\begin{equation}
\xi_\alpha(f\, t_\alpha^{n+1} )=(n+1)\bar{f}\, t_\alpha^n \,\,\,\otimes  q^\ast(\pi^\ast\sigma \otimes \widehat{v}_\alpha), \\
\end{equation}
where $\bar{f}\in \sO_{U_\alpha\times \fM}/\I_n^{\alpha}$. From $\I_k^\alpha=\I_{k+1}^\alpha=\I_{k+2}^\alpha=\cdots$, the $\xi_\alpha$ vanishes on the complement of $S_0\times \fM$. The counterpart of the section for $X_\infty$ is defined in the same form with $\widehat{v}$ substituted by $v$ and it also vanishes away from $S_\infty\times\fM$.\\

 The morphism glued over different covers. Let $g_{\alpha\beta}$ be the coordinate transform:
$$t_\alpha =g_{\alpha\beta} t_{\beta},\,\, v_{\alpha}=g_{\alpha\beta}^{-1} v_{\beta},\,\, \widehat{v}_{\alpha}=g_{\alpha\beta}\widehat{v}_{\beta}.$$
Then both side side of (3.1) are transformed to the map over $U_{\beta}$ multiplied by $g_{\alpha\beta}^{n+1}$. Hence $\xi_\alpha\vert_{\alpha \beta}=\xi_\beta\vert_{\alpha,\beta}$. Hence the two morphisms gives sections of 
$Hom_{p_\ast}(\I_0, \sO_{Z_0}\otimes q^{\ast}K_{X_0})$ and  $Hom(\I_\infty, \sO_{Z_\infty}\otimes q^{\ast}K_X)$, both of which vanishes over the  $(X_0\cap X_\infty)\times\fM$. Thus they glue to give a section 
$$\xi\in Hom_{p_\ast}(\I, \mathcal{O}_{X\times \fM}/\I\otimes q^{\ast}K_X)\cong \Ob^{\vee}.$$

 In (3.1), the weight of $t_\alpha$ and $\widehat{v}_\alpha$ are both $-1$ with all other terms of weight zero. So the total weight of the $\xi$ is $(n+1)-(n+1)=0$ and $\xi$ is equivariant. We have
$$\xi\in (\Ob^{\vee})^{\CC^\ast}\cong (\Ob^{\CC^\ast})^{\vee}.$$

\subsection {Vanishings}
\,\,\,\,Now we prove proposition $3.1$:

\begin{proof}
Let $x$ be an arbitrary geometric point in $\fM$ that corresponds to an ideal sheaf
$I_x$ of a $\CC^{\ast} $ invariant subscheme $Z_x$ of $X$. Denote $C_0=C$ in $S_0$ and $C_\infty$ the same curve in $S_\infty$. Since the homology class of $Z_x$ has the horizontal component represented by an effective curve that is not multiple
of canonical class, there exists some point $w$ in the $Z_x\cap S_i$ for some $i=0$ or $\infty$ such that $w$ is not in $C_0\cup C_{\infty}$ and the local ring $\sO_{Z_x\cap S_i,x}$ of dimension one. Without loss of generality we assume $w$ lies in a chart $U_{\alpha}\subset S_0$. We assume further $U_\alpha$ does not intersect $C_0$ so that $\sigma$ is nondegenerate over $U_\alpha$. Denote $I_n:=\I_n^{\alpha}|_x$. For $n\geq 0$, $f\in I_{n+1}$,
\begin{eqnarray*}
\xi_\alpha|_x: \oplus I_n t_\alpha ^n \lra \oplus(\sO_{U_\alpha\times \fM}/ I_n)\, t_\alpha ^n\otimes K_{X_{\alpha}},\\
\xi_\alpha(f\, t_\alpha^{n+1} )=(n+1)\bar{f}\, t_\alpha^n \,\,\,\otimes  (\pi^\ast\sigma \otimes \widehat{v}_\alpha).
\end{eqnarray*}
 Suppose $\xi_\alpha|_x$ is zero homomorphism. Since the two-form $\sigma_{\alpha}$ is nonzero over $U_\alpha$ there is $I_0\supset I_1\supset I_2\cdots$. By the argument in previous section there is
$I_0=I_1=I_2=\cdots$ and hence 
$$\sO_{Z_x}|_{X_\alpha}\cong (\sO_{U_\alpha}/I_0)[t_\alpha]=\sO_{Z_x\cap S_0}[t_\alpha].$$
Since $Z_x\cap S_0$ is of dimension one at $w$, the dimension of $Z_x\cap\alpha$ at $w$ is equal to two. This contradicts to that $Z_x\in I_n(X,\beta)$ is an one dimensional subscheme of $X$.

Since  $Ext^3_X(\sO_{Z_x}, I_x)^{\vee}\cong Hom_X(I_x, \sO_{Z_x} \otimes K_X)$, the above argument shows that the restriction $$\xi|_x: Ext^3_{p_\ast}(\sO_Z, \I)|_x=Ext^3_X(\sO_{Z_x},I_x)\to \CC$$
is a surjective homomorphism. Hence as a cosection of $\Ob^{\CC^*}$, $\xi$ nonzero at every geometry point $x\in I_n(X,\beta)$. By corollary (2.4)
\begin{eqnarray*}
[I_n(X, \beta) ^{\CC^{\ast}}]^{virt} =0.
\end{eqnarray*}
By the virtual localization formula, $[I_n(X,\beta)]^{vir}=0$.
 \end{proof}
 This shows that the Donaldson-Thomas invariants of the
 $\Po$ scroll are zero whenever $\beta_h$ is not multiple
 of $[C]$. From this one concludes that counting of ideal sheaves
 of a $\Po$ scroll over a $K3$ surface or an abelian surface
 is always zero.

  \section{\bf General three-fold with a two-form}

 We extend the two-form localiztion of DT invariants to general three-folds.
 Let $X$ be a general 3-fold, and $\sigma$ a holomorphic two
form on $X$. Note here $X$ needs not to be a $\Po$ scroll over a surface. Assume $\beta$ is in $H_2(X,\ZZ)$. For
any closed point  $x\in \cM:=I_n(X,\beta)$ that corresponds to a subscheme $Z\subset X$, there is an exact sequence:
\begin{eqnarray*}
0\longrightarrow T_Z \longrightarrow T_X\vert_{Z} \longrightarrow \mathcal{N}_{Z\subset X}
\end{eqnarray*}
Tensor $K_X$ and take global sections,
\begin{eqnarray*}
\phi_{Z}: \Gamma(X,T_Z\otimes K_X)\longrightarrow \Gamma(X,T_X\vert_{Z}\otimes K_X)=\Gamma(X, \Omega^2_X\vert_{Z})
\end{eqnarray*}

\begin{defi}
The degeneracy loci of the cosection $\xi$ is the set of all $x\in \cM$ which corresponds to subschemes
$Z$ such that the restriction $\sigma \vert_Z$ lies in the image of $\phi_Z.$
\end{defi}

\begin{prop}
If $\xi$ restricts to zero at $x\in I_n(X,\beta)$ then $x$ is in the degeneracy loci. If the degeneracy loci is empty then $[I_n(X,\beta)]^{vir}=0$. In this case the DT invariants vanishes
$$Z_{DT}(X,\prod \tau_{q_i}(\gamma_i))_{\beta,n}=0\,\, \text{for all descendants}\, \gamma_i.$$
\end{prop}

\begin{proof}
Here $\Z$ is the universal family of schemes in $X\times \cM$ parameterized by $I_n(X,\beta)$. By Serre duality one has over $\cM$:
\begin{eqnarray*}
\Ob^{\vee}&=&\mathcal{H}om_{\sO_\cM}(Ext^3_{p_\ast}(\sO_{X\times\cM}/\I_{\Z}, \I_{\Z}), \sO_\cM)\cong Hom_{p_\ast}(\I_{\Z}, \sO_{X\times \cM}/\I_{\Z}\otimes q^{\ast}K_X)\\
&\cong&p_\ast( \mathcal{H}om(\I_{\Z}/\I^2_{\Z}, \sO_{X\times \cM}/\I_{\Z})\otimes q^{\ast}K_X)=p_\ast(\mathcal{N}_{\Z\subset X\times \cM}\otimes q^{\ast}K_X).
\end{eqnarray*}
Tensor the sequence
\begin{eqnarray*}
0\longrightarrow T_\Z\longrightarrow q^{\ast}T_X\vert_{\Z} \longrightarrow \mathcal{N}_{\Z\subset X\times \cM}
\end{eqnarray*}
 with $q^{\ast}K_X$ and take pushforward by $p$:
\begin{eqnarray*}
0\longrightarrow p^{\ast}(T_\Z\otimes q^{\ast}{K_X})\longrightarrow p_\ast(q^{\ast}T_X\vert_\Z\otimes q^{\ast}K_X) \longrightarrow p_\ast(\mathcal{N}_{\Z\subset X\times \cM}\otimes q^{\ast}K_X).
\end{eqnarray*}

The middle term $p_\ast(q^{\ast}T_X\vert_\Z\otimes q^{\ast}K_X) $ is isomorphic to
$p_\ast(q^{\ast}\Omega^2_X)|_{\Z}$ and $\sigma$ induced a global section. Denote this
global section of $\Ob^{\vee}$ by $\xi$. Given an arbitrary closed point $x$ in $I_n(X,\beta)$ with $Z=\Z|_x$, there is a commutative diagram:
$$
\begin{CD}
0@>>> p^{\ast}(T_Z\otimes q^{\ast}{K_X})\vert_x @>>> p_\ast(q^{\ast}T_X\vert_\Z\otimes q^{\ast}K_X)\vert_x @>>> p_\ast(\mathcal{N}_{\Z\subset X\times \cM}\otimes q^{\ast}K_X)\vert_x \\
@. @VVV @VVV @VV{\lambda}V \\
0@>>> \Gamma(X,T_Z\otimes K_X) @>{\phi_Z}>> \Gamma(X, \Omega^2_X \vert_Z) @>>> N_{Z\subset X}\otimes q^\ast K_X.
\end{CD}
$$
 The value of $\xi$ at $x$ has its image under $\lambda$ to be $\sigma|_{Z}\in\Gamma(X,\Omega^2_X|_X)$.  If $\xi|_x=0$ then $\sigma|_Z=0$
implies $x$ is in the degeneracy loci. If degeneracy loci is empty the cosection $\xi$ is a surjection and the virtual cycle vanish by corollary (2.4).
\end{proof}

  If $X$ is the $\Po$ scroll, the $\xi$ in section 3 is the same
 as the $\xi$ here, after restricting to the $\CC^{\ast}$ fixed moduli space.\\

 In case that $X=\Po(\sO_S\oplus L)$ as in section 3, the following properties characterize the possible configuration of the components of the subscheme $Z$ from the degeneracy loci. 

\begin{prop} Let $X=\Po(\sO_S\oplus L)$. Suppose $Z$ is in the degeneracy loci of $\xi$. Then $Z$ is a disjoint union of subschemes of two types: the first has reduced part lies in $P(\sO_C\oplus L|_C)$ and the second has supports equal to fibers
over points outside of $C$ with the scheme structure $\CC^{\ast}$ invariant.
\end{prop}

\begin{proof}

  \,\,Let $q$ be a point in $X$ that $Z$ passed by and its projection to the surface is a point $p$ on $S_0$. Pick up an analytic neighborhood $U$ (or etale neighborhood) of $q$ in $X$ which comes from coordinate $x,y$ on $S_0$ and $t$ for the vertical direction. Let the definig ideal of $Z$ in $U$ be $\I=\mathcal{I_Z}$ and $\CC[x,y,t]/\I=A$. Denote $\frac{\partial}{\partial_x}, \frac{\partial}{\partial_y}, \frac{\partial}{\partial_t}$ as the standard tangent vector field (holomorphic) and $dx, dy, dt$ the standard one forms. Since $\sO_Z\vert_U=\CC[x,y,t]/\I=A$, one has the following exact sequence:
 \begin{eqnarray*}
0 \longrightarrow \cM \longrightarrow \Omega^1_{U}\longrightarrow \Omega^1_{Z}\longrightarrow 0,
\end{eqnarray*}
 where $\cM=\<df\>_{f\in \I}$ is the module generated by $\{ df\vert f\in \I\}$ in $\Omega^1_{U}=\CC[x,y,t]dx\oplus \CC[x,y,t]dy\oplus \CC[x,y,t]dt$.
By taking the dual functor $Hom_{\CC[x,y,t]}(\cdot, A)$,

\begin{eqnarray*}
0\longrightarrow T_Z \mapright{\tau_1} A \frac{\partial}{\partial_x}\oplus A\frac{\partial}{\partial_y}\oplus A\frac{\partial}{\partial_t}\mapright{\tau_2} Hom_{\CC[x,y,t]}(\cM, A).\\
\end{eqnarray*}

 By tensoring this sequence with $K_U=\CC[x,y,t]dx\wedge dy \wedge dt$ the map $\tau_1\otimes Id_{K_U}$
 is the same as $\phi_Z\vert_U$ after taking global sections. The assumption that $\sigma\vert_Z$ is in
 the image of $\phi_Z$ implies that the two-form $\sigma\vert_Z\vert_U=dx\wedge dy$ is inside the
 image of $\tau_1\otimes Id_{K_U}$. Clearly this elements comes from the vector field $\frac{\partial}{\partial_t}$
 tensored with the 3-form $dx\wedge dy\wedge dz$ of $K_U$. So the assumption implies $\frac{\partial}{\partial_y}$
 lies in the image of $\tau_1$, which is equivalent to say $\tau_2(\frac{\partial}{\partial_t})=0$.\\

 Now  $\tau_2(\frac{\partial}{\partial_t})(df)=0$ for all $f\in \I$. By writing
 $f=g_0(x,y)+g_1(x,y)t+g_2(x,y)t^2\cdots$ one has 

 \begin{eqnarray*}
0&=&\tau_2(\frac{\partial}{\partial_t})(df)=df(\frac{\partial}{\partial_t})\\
&=&[dg_0+(dg_1)t+(dg_2)t^2+(dg_3)t^3\cdots ](\frac{\partial}{\partial_t})+ [g_1dt+g_2\cdot 2tdt+g_3\cdot 3t^2dt+\cdots](\frac{\partial}{\partial_t})\\
&=&g_1+2g_2t+3g_3t^2+\cdots,
 \end{eqnarray*}

 which implies $g_1=g_2=g_3=\cdots=0$. This is the same as saying that the part of the subscheme
 $Z$ near $q$ is $\CC^{\ast}$-equivariant.
\end{proof}

\begin{coro}
Let $\sigma$ be a holomorphic two-form on $S$ with smooth zero loci $C$. Suppose the horizontal component of $\beta$ in $H^2(S,\ZZ)$ is not a multiple of $C$, then one has ${I}_n(X,\beta)_{\CC^{\ast}}^{vir}=0$ and so ${I}_n(X,\beta)^{vir}=0$
\end{coro}
\begin{proof}
By construction in \cite{Jli} the localized virtual circle could be constructed in the degeneracy loci. The proposition 4.2 shows for $Z$ in degeneracy loci the horizontal component of $Z$ is a multiple of $C$. 
\end{proof}

 The above characterizes the behavior of subschemes $Z$ which contributed to the virtual
 cycle. In view of this we will define the the localized Donaldson Thomas invariant for $\Po$ scrolls over $S$. The invariant only depends on the $\Po$ scroll over a small neighborhood of the canonical curve $C$ in $S$.

\section{\bf Two-form localization for $\Po$ scroll}

 \,\,Two-form localization is a localization method applied to smooth varieties equipped with a
 global holomorphic two-form. Most varieties of general types satisfy this condition. The localization consists of two parts. First one shows the contribution of the holomorphic curves vanishes when the curve is away from the degeneration loci of the two-form. 
 Secondly one shows the invariant is the same as the localized invariant defined for an open analytic neighborhood of the degeneracy loci of the two form. The analytic treatment originates from Thom Parker's symplectic approach to compute the Gromov Witten invariants of a $p_g>0$ surface. In \cite{tom} T, Parker and J.H, Lee used the holomorphic two-form to perturb the original integrable complex structure $J$ to another non-integrable one $J^{\prime}$ so that any pseudo-holomorphic curve with respect to $J^{\prime}$ is a holomorphic curve with respect to original $J$ and $\textbf{\textit{lies in the degeneration loci of the two-form}}$.  In surface case this implies the only curves that contribute to the Gromov-Witten invariants are the degeneration
 loci or its multiples, and one can compute the contribution from the neighborhood of the loci, which is analytically
 the normal bundle of the curve.\\

 J. Li and Y.H. Kiem constructs the localization for the same problem from algebraic side.
 In \cite{Jli} they do not perturb the complex structure but instead build a map from the
 obstruction sheaf to the structure sheaf over the moduli of stable maps. Then the intrinsic cone lies in
 the kernel of the map (lemma (\ref{cosection})) and the the zero locus of cosection parametrizes those mapping to the canonical
 curve $C$. They pick a metric on the bundle resolution of the obstruction sheaf.  The metric induces an $C^{\infty}$ inverse of the cosection which intersects the cone only above the degeneration loci. They use the Gysin map to
 build the localization scheme.\\

 The same method would apply to higher dimensional case. The Gromov-Witten invariants
 of a variety $M$ with a two-form $\sigma$ are contributed only by those stable maps over which the
 corresponding cosection $Ob\rightarrow \sO_M$ is zero; and these stable maps are $\bf{almost}$
 those maps to the zero loci of $\sigma$. According to MNOP conjectures on GW-DT correspondence one would
 expect this applies to DT theory for any three-fold with a two-form, for example a $\Po$ scroll over a $p_g>0$
 surface. In DT case the vanishing property is proved in previous sections. However the analogy of localization to the $\theta$-neighborhood does not follows the GW case directly because of the $\Po$-fiber class. The point where the cosection of
obstruction sheaf degenerates corresponds to subschemes which may have the fiber components roaming far away from the canonical curve. We overcome this difficulty and define the localized Donaldson Thomas invariant of $\Po$-scroll over surface with $p_g>0$. We show that it depends only on the neighborhood of the $\Po$ scroll over the canonical curve.

\subsection{Localize to $\theta$ neighborhood}

 In \cite{Jli} the localization principle asserts that the whole GW-invariants of the surface $S$ with canonical
curve $C_g$ can be completely determined by the "theta-neighborhood" of $C_g$. The ''theta-neighborhood`` of $C_g$ in $S$ is isomorphic to a theta line bundle of $C_g$. The deformation class of the curve with the theta line bundle depend on $K_S\cdot K_S$ and $(p_g\mod 2)$ (ref \cite{tom} and \cite{Jli}). We will prove the following:\\

\noindent{\bf Claim}:{\,\, \sl Let $S$ be a surface with $p_g>0$ and Let $X$ be the compactification of a line bundle
$L$ over $S$. Then the cycle $[I_n(X,\beta)]^{vir}$ can be constructed from the $\theta$ neighborhood of $L$ over
the canonical curve $C$.}\\

  The construction is as follows. First one picks up a neighborhood $U$ of the curve $C$ in $S$ which is
 small enough so that it is analytically isomorphic to the collection of points on the normal bundle
 of $C$ ins $S$ whose distance from $C$ is $1$ under some metric $g$. Take a smaller neighborhood $V$
 which consists of points with distance $1/2$ from the $C$ and denote it by $V$. \\

 Given an arbitrary nonnegative integer $k$, let $I_n(U,V,\beta-kF) $ be the analytic open subset
 of $I_n(X,\beta-kF)$ which consists of subschemes $Z$ satisfying the following two condiitons:\\
  $(a)$: \textit{The support of $Z$ is in} $\pi^{-1}(U)$\\
  $(b)$: \textit{Every connected component of} $Z$ \textit{that intersects} $\pi^{-1}(U-V)$
  \textit{has its ideal sheaf equals} $\pi^{\ast}I_T$\textit{ where }$T$ \textit{is a subscheme of points in }$U$
  \textit{and} $I_T$ \textit{its ideal sheaf in }$\sO_U$; \textit
  {this is equivalent to say the part of the subscheme is $\CC^{\ast}$-equivariant}.\\

  Denote this sequence of open subsets by $B_{k}$ where $k=0, 1,2,3...$. On the other hand,
  let another analytic open set of $I_n(X,\beta)$ be the collection of  subschemes $Z$ that only
  satisfy the following conditions analogous to (b) above:\\
  $(b^{\prime})$: \textit{Every connected component of} $Z$ \textit{that intersects} $\pi^{-1}(S-V)$
  \textit{has its ideal sheaf equals} $\pi^{\ast}I_T$ \textit{where} $T$ \textit{is a subscheme of points in}
  $S$ \textit{and} $I_T$ \textit{its ideal sheaf in} $\sO_S$.\\

  Denote this open subset of $I_n(X,\beta)$ by $B$. By proposition 4.2 and the fact that any small
  perturbation of a $\CC^{\ast}$-equivariant scheme is still $\CC^{\ast}$-equivariant, one deduces that
  $B$ is an open neighborhood of the degeneracy loci of the cosection given in $\xi$. From the
  two-form localization principle (\cite{Jli1}\cite{Jli}) the virtual cycle of $I_n(X,\beta)$
  can be constructed by intersecting (after small perturbation) a $C^{\infty}$ section of the
  obstruction sheaf over $B$ with the normal cone where the section must be a lift of $\xi$
  away from degeneracy loci. The same procedure does not work for $B_{k}$ but after being modified as
  follows, gives us a localization method for this version.\\

   There is a map of sets (even not continuous) from $B$ to disjoint union of $B_k$:
  \begin{eqnarray*}
  r_U:B\longrightarrow \bigsqcup_k B_k,
  \end{eqnarray*}
  acts on a subscheme by forgetting its part outside $\pi^{-1}(U)$.
  Such a map is well defined because of the condition $(b)$ and $(b)^{\prime}$. It is not continuous
  because one can pick a $\CC^{\ast}$-equivariant fiber component over a point $q$ on $S$ and let $q$
  approach $\partial U$ in $S$. To make $r_U$ a continuous map one needs to add some open sets in the
  topology of $\bigsqcup_k B_k$. There is clearly a canonical choice of such topology because each
  neighborhood of any point $Z$ in $B_k$ is a product of the $\CC^{\ast}$ fibers configuration
  (which is a smooth manifold as the hilbert scheme of points on the surface $S$) and the perturbation of
  subschemes in $\pi^{-1}(U)$ (which is not smooth but an analytic space). One would refer to
  $\bigsqcup_k B_k$ as the same disjoint union but with this new canonical topology. Then the map $r_U$
  is continuous. \\

   We define ''topological-analytical mixed space" as those glued by charts which
  are products of a topological space $T$ and analytic spaces $W$, and restrict the gluing homeomorphism
  $g_{ij}$ from $U_i=T_1\times W_1$ to $U_j=T_2\times W_2$ to be of the form

  \begin{eqnarray*}
  &W_1=W_0\times \widehat{W};\,\,\,\, g_{ij}=u_{ij}\times w_{ij}\\
&w_{ij}: W_0\mapright{~} W_2\\
&u_{ij}:T_1\times \widehat{W}\mapright{~} W_1,
\end{eqnarray*}
 $Z_{DT}(X, \prod\tau_{0}(\gamma_i))_\beta$
 where $w_{i,j}$ is biholomorphic and $u_{i,j}$ is homeomorphic. A coherent sheaf on this kind of space
 is defined by gluing sheaves over each analytic components $W$ by $w_{ij}$. By definition the space
 $\bigsqcup_k B_k$ is automatically such a "topological-analytical mixed space", which will still be denoted by $B$.
  On the other hand the obstruction sheaf comes from deforming ideal sheaves in $\pi^{-1}V$ and hence
  is a coherent sheaf over this space. It is then by definition that the obstruction sheaf of $B$ is
  the same os the pull back of the obstruction sheaf of $\bigsqcup_k B_k$ by $r_U^{\ast}$. The cosection
   is also compatible and one can check \textit{the degeneracy loci of $\bigsqcup_k B_k$ is compact}. Now we
   can prove the claim:

 \begin{defiprop}
 For a theta neighborhood $U$ of a smooth genus curve $C$ and a line bundle $L$ on $C$. Given
 $\beta \in H_{\ast}(P(L\oplus \sO_C))$, the localized virtual fundamental class,
$[B]^{vir}\in H^{BM}_{\ast}(B)$ is defined to be the intersection of the intrinsic normal cone with the zero
section inside the body of the obstruction sheaf $Ob$ over $B$. The intersection is constructed by perturbing the
zero section nearby the degeneracy loci and then intersect in the same way as (2.5) in \cite{Jli}. If $U$ is the
neighborhood of the canonical curve $C$ in $S$, then
$r_U^{\ast}([B]^{vir})=[I_n(X,\beta)]^{vir}$.
\end{defiprop}

  Denote $Z\subset I_n(X,\beta)\times X$ the universal subscheme and $I$ the universal ideal sheaf.
  Let $\widehat{U}$ be the one point compactification of $U$. Denote the point in $\widehat{U}-U$ by $p_0$.
  Over $B$ there is a universal ideal sheaf $\widehat{I}$ and a universal "subscheme"
  $\widehat{\Z}$ inside the space
  $B\times \pi^{-1}(\widehat{U})$, where one realizes the fibers $F_p$ converge to the
  same fiber $\pi^{-1}(p_0)$ when $p$ converges to $p_0$.
  For $\hat{\gamma}\in H^{\ast}(\pi^{-1}(\widehat{U}))$, let $ch_{k+2}(\hat{\gamma})$ be the
  following operation on the homology of $B$:

  \begin{eqnarray*}
  ch_{k+2}(\hat{\gamma}): H_{\ast}(B,Q)\rightarrow H_{\ast-2k+2-l}(B,Q),\\
  \,\,\,ch_{k+2}(\hat{\gamma})(\alpha)=\pi_{1 \ast}(ch_{k+2}(I)\cdot \pi_2^{\ast}(\hat{\gamma})\cap \pi^{\ast}_1(\alpha)).
  \end{eqnarray*}

  Define the DT invariant:
\begin{defi}
\begin{equation*}
\<\tilde{\tau}_{k_1}(\hat{\gamma}_1)\ldots \tilde{\tau}_{k_r}(\hat{\gamma}_{r})\>_{n,\beta}^{\pi^{-1}(U),
loc}:=\prod_{i=1}^r (-1)^{k_i+1}ch_{k_{i}+2}(\hat{\gamma}_i)([B]^{vir})
\end{equation*}
\end{defi}

 Then
 \begin{prop}
 Suppose $L=\sO(mC)$ for some $m$. Further assume $\beta$ is in $H_{\ast}(P(L\oplus \sO_C))$ and
 $\{\hat{\gamma}_i\}$ are Poincar\'e dual of homology classes $\check{\gamma_i}\in P(L\oplus \sO_C)$ in $\pi^{-1}(\widehat{U})$.
 Then (prime fields) localized DT invariant
 $\<\tilde{\tau}_{0}(\hat{\gamma}_1)\ldots \tilde{\tau}_{0}(\hat{\gamma}_{r})\>_{n,\beta}^{\pi^{-1}(U), loc}$
 are the same as the original invariant
 $\<\tilde{\tau}_{0}(\hat{\gamma_1})\ldots \tilde{\tau}_{0}(\hat{\gamma_{r}})\>_{n,\beta}^{X}$, where ${\hat{\gamma_i}}$ are
 Poincar\'e dual of $\check{\gamma_i}$ in $X$.
\end{prop}

 \begin{proof}
 Since $L=\sO(mC)$ is trivial outside $C$, the space $\pi^{-1}(U)$ can be canonically compactified by adding one
 additional $\Po$. The space is a $\Po$ bundle over $\widehat{U}$ and would be denoted by $\pi^{-1}(\widehat{U})$. There
 is a continuous map $\mu:X\rightarrow \pi^{-1}(\widehat{U})$ by identifying all $\Po$ outside $U$ to the single additional
 $\Po$. One can take the Poincar\'e dual of homology classes $\check{\gamma_i}$ from $P(L\oplus \sO_C)$ either in $X$ or in
 $\pi^{-1}(\widehat{U})$, and the resulting cohomology classes ,$\{\hat{\gamma}_i\}$ and $\{\gamma_i\}$ would corresponds to each other under the
 map $\mu^{\ast}$, that is $\gamma_i=\mu^{\ast}(\hat{\gamma}_i)$.\\

 The following diagram commutes:
\begin{eqnarray*}
I_n(X,\beta) \leftarrow^{\pi} &Z\subset I_n(X,\beta)\times X &\mapright{\pi_2} X\\
\downarrow r_U\,\,\,\,\,\,\,\,\,\,\,\, &\downarrow \phi & \,\,\,\,\,\,\,\,\,\,\,\,\,\downarrow \mu\\
B\,\,\,\,\,\,\,\,\,\,\,\,\leftarrow &\widehat{Z}\subset B\times \widehat{U}&\mapright{\pi_2} \widehat{U},
\end{eqnarray*}
where $\phi=r_U\times \mu$.
Let
\begin{eqnarray*}
\varsigma_i&=&(-1)^{i+1}ch_{2}(\hat{\gamma}_{r-i})\circ \cdots \circ ch_{2}(\hat{\gamma}_r)([I_n(U,\beta)]^{vir} )\\
\hat{\varsigma}_i&=&(-1)^{i+1}ch_{2}(\gamma_{r-i})\circ \cdots \circ ch_{2}(\gamma_r)([B]^{vir} ).
\end{eqnarray*}
  Use induction and assume
$\varsigma_{i}=\gamma_U^{\ast}(\hat{\varsigma}_{i})$, we need to show $\varsigma_{i+1}=\gamma_U^{\ast}(\hat{\varsigma}_{i+1})$.
\begin{eqnarray*}
\varsigma_{i+1}&=&-\pi_{1\ast}(ch_2(I).\pi_2^{\ast}\gamma_{i+1}\cap \pi_1^{\ast}\varsigma_i)\\
&=&-\pi_{1\ast}(ch_2(I).\pi_2^{\ast}\mu^{\ast}\widehat{\gamma}_{i+1}\cap\pi_1^{\ast}\gamma_U^{\ast}\widehat{\varsigma_i})\\
&=&-\pi_{1\ast}(ch_2(I).\phi^{\ast}(\pi_2^{\ast}\widehat{\gamma}_{i+1}\cap\pi_1^{\ast}\widehat{\varsigma}_i))\\
&=&-\pi_{1\ast}([Z]\phi^{\ast}(\pi_2^{\ast}\widehat{\gamma}_{i+1}\cap\pi_1^{\ast}\widehat{\varsigma}_i))\\
&=&-\pi_{1\ast}\phi^{\ast}(\phi_{\ast}[Z].\pi_2^{\ast}(\widehat{\gamma}_{i+1}\cap\pi_1^{\ast}\widehat{\varsigma}_i))\\
&=&-\pi_{1\ast}\phi^{\ast}([\widehat{Z}].\pi_2^{\ast}(\widehat{\gamma}_{i+1}\cap\pi_1^{\ast}\widehat{\varsigma}_i))\\
&=&-\pi_{1\ast}\pi^{\ast}(ch_2(\widehat{I}).\pi_2^{\ast}(\widehat{\gamma}_{i+1}\cap\pi_1^{\ast}\widehat{\varsigma}_i)\\
&=&-\gamma_U^{\ast}\pi_{1\ast}(ch_2(\widehat{I}).\pi_2^({\ast}\widehat{\gamma}_{i+1}\cap\pi_1^{\ast}\widehat{\varsigma}_i)\\
&=&\gamma_U^{\ast}(\widehat{\varsigma}_{i+1})
\end{eqnarray*}
\end{proof}

 Here we expect the prime field condition can be dropped and apply to all descendant insertions but now we can only prove
 for prime field cases because we don't know the Poincar\'e dual of chern characters $ch_{k+2}(I)$ with $k>0$. On the other
 hand, if one starts with a $\theta$ line bundle, said $\theta$, of a smooth proper curve $C$ and choose a line bundle
 $\pi:L\rightarrow C$, then the localized Donaldson Thomas invariant of the open three-fold $\pi^{-1}(\theta)$ is defined
 in the same form, where one sets $U=\theta$ and pick arbitrary $V$. (The choice of $V$ will not affect the degeneracy loci and the localized
 virtual cycle. see \cite{Jli}). The number
 $\<\tilde{\tau}_{k_1}(\gamma_1)\ldots \tilde{\tau}_{k_r}(\gamma_{r})\>_{n,\beta}^{\pi^{-1}(U), loc}$
 is defined when $\gamma_i$s are Poincar\'e dual of homology classes from $H_{\ast}(P(L\oplus \sO_C))$ in
 $\pi^{-1}(\widehat{U})$.

\subsection{Deformation invariance}

 Let $S$ be a smooth open three-fold smooth over the affine line $T=Spec \,\CC$. $C=\bigcup_{t\in T} C_t$ is
 a smooth family of smooth compact curves inside $S$. Assume $S$ is contractible to $C$ and let $\widehat{S}$ be
 the fiberwise one point compactification of $S/T$. Further assume there is a relative holomorphic two-form
 $\sigma \in\Omega^2_{S/T}$ with degeneracy loci on each fiber equal to
 $C_t$. Also let $L=\sO(mC)$ be the line bundle defined by the divisor $mC$ on $S$ and the compactification of
 $L$ over $S$ to be the smooth four-fold $W$. Denote the corresponding line bundle over $\widehat{S}$ by $\widehat{W}$.
 Let the projection $W\rightarrow S$ by $\pi$. Pick up a smaller
 neighborhood $V$ of $C$ in $W$. One applies the above construction familywise to get a family of moduli space
 $B=\bigcup_{t\in T}{B_t}$. Let $\widehat{Z}\subset B\times \widehat{W}$ be the universal family
 of generalized subschemes parametrized by $B$. Here "generalized" means one mark the "nonreduced" structure of a subscheme
 along $\pi^{-1}(p_0)$ only by multiplicities. The global obstruction sheaf $Ob_{B}$ is equal to
 $Ext^2_{\pi_1}(I, I)_0$, where $\pi_1: B\times \widehat{W} \rightarrow B$ is the projection. The relative obstruction sheaf $Ob_{B/T}$
 would then be $Ext^2_{\widehat{\pi}_1}(I, I)_0$, where $\widehat{\pi}_1:B\times_T \widehat{W}\rightarrow B$ is also the
 projection. Similar to the ordinary case there is an exact sequence:
 \begin{equation}\lab{def}
\sO_B\mapright{\delta} Ext^2_{\widehat{\pi_1}}(I, I)_0 \longrightarrow Ext^2_{\pi_1}(I, I)_0.
 \end{equation}
 The sequence is constructed by Richard Thomas in lemma (3.42) in \cite{Casson}. Here we reproduce the proof via the language
 of derived category because it will be used in the proof of deformation invariance later.

\begin{lemm}
 Suppose $\iota:D\subset Y$ is a Cartier divisor in a quasi-projective scheme $Y$, with normal bundle $\nu=\sO_D(D)$.
 Then for coherent sheaves $\mathcal{E}$ and $\mathcal{F}$ on $D$ there is a long exact sequence
\begin{equation*}
\rightarrow Ext^i_D(\mathcal{E},\mathcal{F})\rightarrow
Ext^i_Y(\iota_{\ast}\mathcal{E},\iota_{\ast}\mathcal{F})\rightarrow
Ext^{i-1}_D(\mathcal{E},\mathcal{F}\otimes\nu)\mapright{\delta} Ext^{i+1}_D(\mathcal{E},\mathcal{F})
\end{equation*}
\end{lemm}

\begin{proof}
 First we check the exactness of the sequence on $D$
 \begin{equation}\lab{seq}
 0\longrightarrow E[1]\otimes_D \sO_D(-D)\longrightarrow L\iota^{\ast}(\iota_{\ast}E)\longrightarrow E\longrightarrow 0.
 \end{equation}
 Here $L\iota^{\ast}$ is the derived pullback and $\iota_{\ast}$ is the same as the derived pushforward
 because $\iota$ is an inclusion. Take a resolution $E.\rightarrow \iota_{\ast}E \rightarrow 0$ on $Y$,
 and tensor it with $\sO_D$ it becomes $E.\vert_D$ on $D$ whose cohomology
 computes $H^{\ast}(D, L\iota^{\ast}(\iota_{\ast}E))$. On the other hand the complex $E.\vert_D$ viewed as a complex on
 $Y$ is the same as $\iota_{\ast}E\otimes_Y^L\sO_D$ in the derive category over $Y$. Hence there is
 $$H^{\ast}(L\iota^{\ast}(\iota_{\ast}E))=\iota^{\ast}\iota_{\ast}H^{\ast}(L\iota^{\ast}(\iota_{\ast}E))=
\iota^{\ast}(H^{\ast}(\iota_{\ast}E\otimes^L_Y\sO_D)).$$ 
 Take the canonical resolution of
$\sO_D$ on $Y$: $0\rightarrow\sO_Y(-D)\rightarrow \sO_Y\rightarrow
\sO_D\rightarrow 0$. Take (nonderived) tensor of it with $\iota_{\ast}E$
\begin{equation*}
\iota_{\ast}E\otimes_Y \sO_Y(-D) \mapright{g} \iota_{\ast}E \mapright{f} \iota_{\ast}E\otimes
\sO_D
\end{equation*}
It is easy to check $g=0$ and $f$ is an isomorphism. Since the complex $[\iota_{\ast}E\otimes_Y \sO_Y(-D)
\mapright{g} \iota_{\ast}E]$ is how we define $\iota{\ast}E\otimes_Y^L\sO_D$, one has
\begin{eqnarray*}
H^{\ast}(L_i^{\ast}(\iota_{\ast}E))=\iota^{\ast}(H^{\ast}(\iota_{\ast}E\otimes^L_Y\sO_D))
=& & 0 \,\,\,\,\,\,\,\,\,\,\,\,\,\,\,\,\,\,\,\,\,\,\,\,\,\,\,\,\,\,\,\,\,\,\,\,\,\,\,\,\,\,\,         \ast \neq 0,-1\\
 & & E \,\,\,\,\, \,\,\,\,\,\,\,\,\,\,\,\,\,\,\,\,\,\,\,\,\,\,\,\,\,\,\,\,\,\,\,\,\,\,\,\,          \ast =0\\
 & & E\otimes\sO_D(-D)                                         \,\,\,\,\,\,\,\,\,\,          \ast =-1.
\end{eqnarray*}
So the exact sequence (triangle) (\ref{seq}) follows.\\
 For the lemma one simply takes $RHom_D(\cdot,F)$ to the sequence (\ref{seq}) and use 
 \begin{eqnarray*}
 &RHom_D(L\iota^{\ast}(\iota_{\ast}E),F)[i]& =\,\,\,\,RHom_X(\iota_{\ast}E,R\iota_{\ast}F)[i]\\
                   &\parallel& \,\,\,\,\,\,\,\,\,\,\,\,\,\,\,\,\,\,\,\,\,\,\,\,\,\,\,\,\,\,\,\,\,\,\, \parallel \\
 &Ext^i_D(L\iota^{\ast}(\iota_{\ast}E),F)& =\,\,\,\,\,\,\,\,\,\,Ext^i_X(\iota_{\ast}E,\iota_{\ast}F)
 \end{eqnarray*}

\end{proof}

 The sequence (\ref{def}) follows by applying the lemma to the divisor
 $D=B\times_T \widehat{W}\subset B\times \widehat{W}=Y$ relative to $B$. There is a criterion
 for the flatness of the localized virtual fundamental classes in $H_{\ast}^{BM}(B_{t})$ given by \cite{Jli}.
 The flatness follows if the family cosection $Ext^2_{\widehat{\pi_1}}(I, I)_0\rightarrow \sO_B$ lifts to
 $Ext^2_{\pi_1}(I, I)_0\rightarrow \sO_B$. The lifting exists if the composition of $\delta$
  in (\ref{def}) with the cosection $\xi$ is zero. We will prove it by interpreting it as the action of
  Kodaira Spencer class on the subscheme,\\

\noindent{\bf Claim}:{\sl \,\,\,\,\,The composition $\xi \circ \delta$ is zero.}\\

 It is enough to check the restriction of  $\xi \circ \delta$ at some $p\in B=\bigcup B_t$. Without loss of
generality assume $t=0$. Let $Z$ be a subscheme of $W_0$ corresponds to $p$. Suppose the ideal sheaf of $Z$ in
$W_t$ is $\I$ and the ideal sheaf of $Z$ in $W$ is $I^{\prime}$. The conormal
sheaves of $Z$ in $W_t$ and $W$ fit into an exact sequence:
\begin{equation}\lab{5.3}
0 \rightarrow \sO_Z \rightarrow \I/\I^2 \rightarrow I/I^2 \rightarrow 0
\end{equation}
It gives an element in $\kappa_Z\in Ext^1_{W_0}(I/I^2,\sO_Z)$. Also denote the extension class of the
sequence
\begin{equation}
0\rightarrow I \rightarrow \sO_{W_0} \rightarrow \sO_Z \rightarrow 0
\end{equation}

 by $\varsigma$. There is Yoneda product

\begin{equation*}
Ext^1_{W_0}(I/I^2,\sO_Z)\times Ext^1_{W_0}(\sO_Z, I)\rightarrow Ext^2(I/I^2, I).
\end{equation*}

and a canonical restriction $Ext_{W_0}^2(I/I^2,I)\rightarrow Ext^2_{W_0}(I, I)$.

\begin{lemm}\lab{Yoneda}
The Yoneda product of $(\kappa_Z,\varsigma)$  has its image in  $Ext^2_{W_0}(I, I)$ the same as $\delta(1)$.
\end{lemm}

\begin{proof}

Let the complex \,\,$\sO_{W}\mapright{\cdot t} \I$\, be denoted by $C^{\cdot}$. Tensoring
\,\,$0\rightarrow \sO_W(-W_0)\rightarrow \sO_W \rightarrow \sO_{W_0}\rightarrow 0$\, with
\,\,$\I$\, one has \,\,$L\iota^{\ast}[i](\I)=Tor^i_{W}(\I,\sO_{W_0})=0$ for
$i>1$.\, This shows the exact sequence \,\,$0\rightarrow \sO_{W}\mapright{\cdot t} \I
\rightarrow I \rightarrow 0$\, is a resolution of \,$I$ on \,$W$ that can be used to compute the derived functor
of \,$\iota^{\ast}(I)$, or equivalently, the complex $C^{\cdot}\vert_{W_0}$ is isomorphic to
$L\iota^{\ast}(\iota_{\ast}I)$. So $\delta(1)$ is also the map $I\rightarrow I[2]$ \,induced from the triangle
\,\,$0\rightarrow I[1] \rightarrow
C^{\cdot}\vert_{W_0} \rightarrow I \rightarrow 0$.

On the other hand, combine the (\ref{seq}) and (\ref{5.3}) one has:
\begin{equation*}
 0\rightarrow I \rightarrow \sO_{W_0}\mapright{\cdot t}
\I\vert_{W_0}=\I/\I^2 \rightarrow I/I^2\rightarrow 0.\\
 \end{equation*}

Denote the complex $\sO_{W_0}\mapright{\cdot t} \I\vert_{W_0}$ by $\tilde{C}^{\cdot}$. The
element $(\kappa_Z\vee\varsigma)$ under the Yoneda product is the same as the the map $I/I^2\rightarrow I[2]$
from
the exact triangle $0\rightarrow I \rightarrow \tilde{C}^{\cdot} \rightarrow I/I^2 \rightarrow 0$\\

 The above two sequences fit into the following diagram:

\begin{eqnarray*}
0 \rightarrow   I[1] \rightarrow  &\sO_{W_0}&  \mapright{\cdot t} \,\,\,\, \I\vert_{W_0}
\rightarrow
I \rightarrow 0 \\
  &\downarrow&\,\,\,\,\,\,\,\,\,\,\,\,\,\,\,\,\,\, \downarrow  \\
 0\rightarrow  I[1] \rightarrow &\sO_{W_0}& \mapright{\cdot t}
\,\,\,\,\I\vert_{Z} \rightarrow I/I^2 \rightarrow 0.
\end{eqnarray*}

from which the lemma follows.

\end{proof}

\begin{prop}\lab{5.6}
The composition $\xi \circ \delta$ is zero. Hence the localized virtual fundamental classes $[B_t]^{vir}$ are
constant in $t$ as classes in $H_{\ast}(B_{t}(\xi))$, where $B_{t}(\xi)$ is the degeneracy loci of the
cosection $\xi$ on $B_t$.
\end{prop}

\begin{proof}
In the proof we omit the base which all the cohomology are taken over and always set it to be $W_0$. Following
the notation the previous lemma, Let $\kappa\in H^1(W_0, T_{W_0})$ be the Kodaira Spencer class of $W_0$ in
$W_t$. It is clear that the image of $\kappa$ under the sequence of maps

\begin{eqnarray*}
&&H^1(W_0, T_{W_0}) \rightarrow H^1(Z,T_{W_0}\vert_Z) \rightarrow H^1(Z, N_{Z\subset W_0})\\
&&=H^1(W_0,\mathcal{H}om(I/I^2,\sO_Z))\rightarrow Ext^1_{W_0}(I/I^2,\sO_Z),
\end{eqnarray*}
is the same $\kappa_Z$. So $\delta(1)$ comes from the composition of the Kodaira Spencer class
$\kappa$ with the canonical element $\varsigma=[\sO_{W_0}]\in Ext^1(\sO_Z,I)$. 

\begin{lemm} \lab{can}
The image of $\delta(1)$ under the map
\begin{equation*}
Ext^2_{W_0}(I,I)\rightarrow Ext^3_{W_0}(\sO/I,I)=\Gamma(W_0, N_{Z\subset W_0}\otimes K_{W_0})^{\vee}
\end{equation*}

acts on the holomorphic two-form

\begin{equation*}
\sigma\in \Gamma(W_0,T_{W_0}\vert_Z\otimes K_{X})\rightarrow \Gamma(W_0,N_{Z\subset W_0}\otimes K_{W_0})
\end{equation*}

by contractions
\begin{equation*}
H^1(W_0,T_{W_0}\vert_Z)\times H^0(W_0,\Omega^2\vert_Z)\rightarrow H^1(W_0,\Omega^1\vert_Z)\mapright{\int} \CC,
\end{equation*}
which is the integral of the $\kappa\vee\sigma\in H^1(W_0, \Omega^1_{W_0})$ over the homology class of $\beta$.
\end{lemm}

\begin{proof}
(The author thanks Professor Brian Conrad for the help about Grothendieck duality.)
(1) Let $\iota:Z \hookrightarrow W_0$ be the inclusion and $\pi: W_0\rightarrow \text{Spec}(\CC)$. Denote $\psi=\iota\circ\pi:Z\rightarrow \text{Spec}(\CC)$. The dualizing complex of $Z$ is $w_Z=\psi^!(\sO_{\text{Spect}(\CC)})=\iota^!(K_{W_0}[3])$. There is 
a canonical map $\Omega^1_Z[1]\rightarrow w_Z$ 
and an integration map 
$$H^0(Z,w_Z)\rightarrow \CC$$ by Grothendieck duality. There is also a map following section 3.5 in \cite{Conrad} 
$$\mathcal{E} xt^2_{W_0}(O_Z, K_{W_0})[1] \rightarrow i^!(K_{W_0}[3])=w_Z$$.

On the other hand, the image of $(\sO_{W_0},\kappa\vee \sigma,\sO_{W_0})$ under the composition of 
\begin{eqnarray*}
Ext^1(\sO_Z,I)\times H^1(\mathcal{H}om(I,\sO_Z)\otimes\Omega^2_{W_0})\times
Ext^1(\sO_Z,I)\rightarrow H^1(\mathcal{E}xt^2(\sO_Z,I\otimes\Omega^2_{W_0}))
\end{eqnarray*}
with $I\otimes\Omega^2_{W_0}\mapright{d\times \text{id}}\Omega^1_{W_0}\otimes \Omega^2_{W_0}=K_{W_0}$
is an element in $H^1(\mathcal{E}xt^2(\sO_Z,K_{W_0}))$. If one maps it further to $H^0(Z, w_Z)$, it follows from naturality that the image would be the same as the image of $\kappa\vee\sigma$ under the map $H^0(W_0,\Omega^1_{W_0}|_Z)\rightarrow H^1(W_0,\Omega^1_Z)\rightarrow H^0(Z,w_Z)$.\\
(2)
The element $\xi \circ \delta(1)$ is obtained as follows. First there is a map:
\begin{eqnarray*}
\mathcal{E} xt^1(\sO_Z,I)\times H^1(\mathcal{H} om(I,\sO_Z))\times
\mathcal{E} xt^1(\sO_Z,I)\rightarrow H^1(\mathcal{E} xt^2(\sO_Z,I))
\end{eqnarray*}

 The image of
$(\sO_{W_0},\kappa,\sO_{W_0})$ ends in $H^1(\mathcal{E}xt^2(\sO_Z,I))$.
Now combining this with:
\begin{equation*}
 \sigma\in \Omega^2_{W_0} \text{\,\,\,\,implies\,\,\,\,} I\lra K_{W_0}
\text{\,\,\,\,implies\,\,\,\,}
\mathcal{E}xt^2(\sO_Z,I)\lra\mathcal{E}xt^2(\sO_Z,K_{W_0}).
\end{equation*}
 gives an element in
$H^1(\mathcal{E}xt^2(\sO_Z,K_{W_0}\otimes\sO_Z))$. By the local to global sequence and dimension
reasons there is
$H^1(\mathcal{E}xt^2(\sO_Z,K_{W_0}\otimes\sO_Z))\rightarrow Ext^3(\sO_Z,\sO_Z\otimes
K_{W_0})=\CC$. The final elemnt in $\CC$ is our $\xi\circ \delta(1)$ by lemma (\ref{Yoneda}).\\

 The claim now follows from the following two diagrams:

\begin{eqnarray*}
&Ext^1(\sO_Z,I)\times H^1(\mathcal{H}om(I,\sO_Z)\otimes\Omega^2_{W_0})\times
Ext^1(\sO_Z,I)&\rightarrow H^1(\mathcal{E}xt^2(\sO_Z,I\otimes\Omega^2_{W_0}))\\
&\uparrow&\,\,\,\,\,\,\,\,\,\,\,\,\,\,\,\,\,\uparrow\\
&\mathcal{E} xt^1(\sO_Z,I))\times H^1(\mathcal{H} om(I,\sO_Z))\times
\mathcal{E} xt^1(\sO_Z,I))&\rightarrow H^1(\mathcal{E} xt^2(\sO_Z,I)),
\end{eqnarray*}
and
\begin{eqnarray*}
&H^1(\mathcal{E}xt^2(\sO_Z,K_{W_0}))&\rightarrow \,\,\,\,\,\,\,\,\,\,\,\,H^1(w_Z)\,\,\,\,\,\,\mapright{\int_Z} \,\,\,\,\,\,\,\,\CC\\
&\downarrow&\,\,\,\,\,\,\,\,\,\,\,\,\,\,\,\,\,\,\,\,\,\,\,\,\,\,\,\,\,\,\,\,\,\,\,\,\,\,\,\,\,\,\,\,\,\,\,\,\,\,\,\,\,\,\,\,\,\,\,\,\,\,\,\,\,\,\,\parallel\\
&Ext^3(\sO_Z, K_{W_0})&\mapright{\backsim}
Hom(\sO_Z,\sO_Z)^{\vee}\mapright{\text{sum}} \,\CC.
\end{eqnarray*}
\end{proof}

 Now since $\beta$ is in $H_{2}(P(L\oplus \sO_C))$ and $\sigma\vert_C=0$ this integral is always zero. So
the map
\begin{equation*}
\delta:\sO_B\mapright{\delta} Ext^2_{\widehat{\pi_1}}(I, I)_0
\end{equation*}
composed with the cosection \,\,$\xi:Ext^2_{\widehat{\pi_1}}(I, I)_0 \rightarrow \sO_B$ \,is zero, and
$\xi$ lifts to 
\begin{equation*}
Ext^2_{\pi_1}(I, I)_0\rightarrow \sO_B.
\end{equation*}
By proposition 2.6 in \cite{Jli} the virtual
cycles $[B_t]^{vir}$ are constant.
\end{proof}

\begin{coro}
Let $S$ be a smooth projective surface with a holomorphic two-form $\sigma$ and the zero loci of $\sigma$ is a
smooth curve $C$. Assume $L=\sO(mC)$ for some $m$ and $X=P(L\oplus \sO_S)$. Let $\beta$ be in
$H_{\ast}(P(L\oplus \sO_C))$ and $\{\gamma_i\}$  Poincar\'e dual of homology insertions
$\check{\gamma_i}\in H_*(X,\RR)$. Then prime fields  DT invariant $\<\tilde{\tau_0}(\gamma_1)\cdots\tilde{\tau}_0(\gamma_r)\>^X_{n,\beta}$ depends only on the following:\\
 (1) genus of $C$ and degree of $L$,\\
 (2) $\chi(\sO_S)\in \ZZ_2$,\\
 (3) homology class (its horizontal and vertical degree) of $\beta$ and $\check{\gamma_i}.$
\end{coro}
\begin{proof}
Suppose there are two surfaces $S_1$ and $S_2$ with all the assumptions. One applies the standard degeneration to
$X_i=P(L_i\oplus \sO_{S_i})$ to the normal bundle of $P(L_i\oplus \sO_{C_i})$ in $X_i$. For the
family the localized DT invariant $\<\tilde{\tau}_{0}(\gamma_1)\ldots
\tilde{\tau}_{0}(\gamma_{r})\>_{n,\beta}^{\pi^{-1}(U\vert_t), loc}$ are constant by proposition \ref{5.6}. Hence one
change the problem into comparing the localized invariants for theta line bundle in each case. Since genus of $C$
and  $\chi(\sO_S))\in Z_2$ are the complete invariants of the deformation class of $theta$ over $C_i$,
with data (1) and (2) the two localized problem has targets three-fold deformation invariant (two-forms extends
to families). By proposition \ref{5.6} again the two localized invariants are the same.
\end{proof}

\noindent{\textbf{Remark}:}
In \cite{Jli} it is shown the Gromov-Witten invariant of a $p_g>0$ surface is completely decided by $g_C$ and $\chi(\sO_S)\in \ZZ_2$ (the theta neighborhood of a canonical curve $C$). By the reduction formula in \cite{TGW} the GW invariant of $P^1$ scroll can be shown to depend only on (1)$g_C$ and $deg\,\,L$, (2)$\chi(\sO_S)\in \ZZ_2$, and (3)$\beta$ and $\check{\gamma_i}$ as in the corollary.  Since the Donaldson-Thomas invariant is defined for three-folds as an analogue of the Donaldson invariant for surfaces, one would expect the Donaldson invariant for bundles of any rank $r$ over a $p_g>0$ surface also depends on $r$, $g_C$ and  $\chi(\sO_S)\in \ZZ_2$ of the surfaces. Another evidence for this besides MNOP conjecture is work of C. Taubes, P. Feenhan and T. Leness that connect GW theory to Seiberg-Witten theory and then to Donaldson theory for certain surfaces.

\section {\bf $g=0$ case}

Now let $M$ be a projective $K3$ surface and $p$ a point on $X$. Consider $S=Bl_pM$ and let the exceptional divisor be $E$. Take $L=\sO(dE)$ on $S$, and we would consider the problem of computing Donaldson Thomas invariants of the three-fold $X=P(\sO\oplus \sO(dE))$ where the curve class is taken to be $\beta=nE$ and n is a nonnegative integer.\\

  \noindent{\bf Example:} {\sl 
Assume $\gamma_{i}\in H_{*}(P(\sO\oplus L\vert_{E}))$,
then the DT partition function of $X$: $Z_{DT}(X,
\prod^{r}_{i=1}\tau_{q_i}(\gamma_{i}))_{\beta}$ is a multiple of the
full Donaldson-Thomas invariants of a toric three-fold. The
three-fold is a $\Po$ bundle over a surface which is blown up of
$P^2$ at one point. So by the virtual localization method one can
derive the partition function of $X$ with arbitrary descendants $\tau_{q_i}(\gamma_{i})$.}\\

Here the case for all other $\gamma_{i}$ can also be computed but the result is slightly more complicated.
The rough algorithm is to degenerate $M$ into a normal crossing of
two rational surfaces glued along a common smooth elliptic curve, and then use
degeneration formula for the corresponding three-fold. The computation of the two relative invariants
can then be reduced to that of absolute ones via standard degeneration.\\

 This is actually an example of the program raised by Raoul Pandharipande and Marc Levine \cite{AC}, namely degenerating the three-fold to toric case and then use degeneration formula to reduce the possible problem to toric DT partitions functions.\\

 Consider a degeneration of the $K3$ surface $M$ to a normal crossing of two rational surfaces, $M1$ and $M2$, where $M_1=P^2$ and $M_2$ is $P^2$ blown up at $18$ points on a smooth elliptic curve $E$. By \cite{Friedman} such a degeneration exists is $N_{E/M_1}\otimes N_{E/M_2}$ is trivial, which can be made by suitable choice of the $18$ points on $E$. One can assume the point $p$ varies (holomorphically) to a point $p_0$ on $M_1$. Let the exceptional curve of blowing up $M_1$ at $p_0$ be $C$, then the line bundle $L$ is also degenerated to the line bundle $L_0=\sO_{S_0}(dC)$ on $S_0$. Note here $L_0$ is $\sO(dC)$ on $S_1=Bl_{p_0}M_1$ and is trivial on $S_2=M_2$. Also denote the compactification of $L$ by $X_0$, and let $X_1=P_{S_1}(\sO\oplus \sO(E_0))$, $X_2=P_{S_2}(\sO\oplus \sO)$. So $X$ degenerates to $X_1\cup X_2$ with normal crossing along $D=E \times \Po$.\\

 Let the normal bundle of $E$ in $M_1$ be $N$, also denote the pull back of $N$ on $D$ by $N$. The compactification of $N$ over $D$ is
 $P_D(\sO\oplus N)$.\\

By the degeneration formula (also \cite{AC}), the DT partition
function of $X$ is a combination of that of $X_1/D$ and $X_2/D$:
\begin{eqnarray*}
Z^{\prime}_{DT}(X, \prod^{r}_{i=1}\tau_{q_i}(\gamma_{i}))_{\beta}=
\sum_{ \beta=\beta_1+\beta_2, \eta} Z^{\prime}_{DT}(X_1/D, \prod_{i\in A}
\tau_{q_{i}} (\gamma_{i}))_{\beta_1,\eta} \cdot \\
\frac{(-1)^{\mid \eta\mid -\textit{l}(\eta)}\vartheta({\eta})}{q^{\mid \eta \mid}}.
Z^{\prime}_{DT}(X_2/D,\prod_{i\in B}\tau_{q_{i}} (\gamma_{i}))_{\beta_2,\eta^{\vee}},
\end{eqnarray*}
where $A\cup B=\{1,2,3,,,r\}$ is fixed, $\vert\eta\vert=\beta_1 \cdot [D]=\beta_2 \cdot [D]$, and  $\eta, \beta_1, \beta_2$ run over all possibilities. The constants $\vartheta({\eta})$ comes from 
\begin{eqnarray*}
[\triangle]=\sum_{\vert\eta\vert=k}(-1)^{k-\textit{l}(\eta)}\vartheta(\eta)C_{\eta}\otimes C_{\eta^{\vee}} \,\,\,\,  \in  H^{\ast}(Hilb(S,k)\times Hilb(S,k), Q).
\end{eqnarray*}
  Since every $\gamma_{i}$ is in $ H_{*}(P(\sO\oplus L\vert_{E}))$ one can fix $A$ to be all insertions and $B$ to be the empty insertion.  From previous section, we can assume $[\beta]=m[C]+n[F]$ where $[F]$ is the fiber class. From the fact that the horizontal curve class does not move out the only possibilities of $\beta_1, \beta_2$ is $\beta_1=\beta_h+n_1[F],\beta_2=n_2[F] (n_1+n_2=n)$.  In this case the intersections $\vert\eta\vert=\beta_1 \cdot [D]=\beta_2 \cdot [D]$ is always zero, so the degeneration formula becomes
\begin{eqnarray*}
Z^{\prime}_{DT}(X, \prod^{r}_{i=1}\tau_{q_i}(\gamma_{i}))_{\beta}=
\sum_{n_1} Z^{\prime}_{DT}(X_1/D, \prod_{i}
\tau_{q_{i}} (\gamma_{i}))_{\beta_1} \cdot
Z^{\prime}_{DT}(X_2/D,1)_{\beta_2}.
\end{eqnarray*}

 Generally for a $\Po$ scroll $X$ over a surface and a curve class $[\beta]=m[\beta_h]+n[F]$ where $[\beta]$ is the horizontal components and $[F]$ is the fiber, the virtual dimension of the moduli space $I_n(X,\beta)$ and $I_n(X/D,\beta)$ are the same as :
 \begin{eqnarray*}
 \int_{\beta} c_1(X)=m\int_{\beta_h} c_1(X) + n\int_F c_1(X)=m\<\beta_h, c_1(X)\> + 2n
 \end{eqnarray*}

 Suppose the cohomology classes $\gamma_{i}$ are from $H^{2d_i}(X,Z)$,  then $Z^{\prime}_{DT}(X, \prod^{r}_{i=1}\tau_{q_i}(\gamma_{i}))_{\beta}$ is nonzero only when 
$$v.d.=m\<\beta_h, c_1(X)\>+2n=\sum^r_{i=1} (q_i-1+d_i).$$ Similarly the relative  $Z^{\prime}_{DT}(X_1/D, \prod_{i}
\tau_{q_{i}} (\gamma_{i}))_{\beta_1} $ is nonzero only when  $v.d.=m\<\beta_h, c_1(X_1)\>_{X_1}+2n_1=\sum^{r}_{i=1} (q_i-1+d_i)$. The only case for them to hold simultaneously is $n_1=n$. So
\begin{eqnarray*}
Z^{\prime}_{DT}(X, \prod^{r}_{i=1}\tau_{q_i}(\gamma_{i}))_{\beta}=
 Z^{\prime}_{DT}(X_1/D, \prod_{i=1}^{r}
\tau_{q_{i}} (\gamma_{i}))_{\beta_1=\beta} \cdot
Z^{\prime}_{DT}(X_2/D,1)_{\beta_2=0},
\end{eqnarray*}

The first relative invariants can be related to absolute invariants by standard degeneration( deformation to normal cone ):

\begin{eqnarray*}
Z^{\prime}_{DT}(X_1, \prod_{i=1}^{r} \tau_{q_{i}}
(\gamma_{i}))_{\beta}&=& Z^{\prime}_{DT}(X_1/D, \prod_{i=1}^{r}
\tau_{q_i}(\gamma_{i}))_{\beta} \cdot
Z^{\prime}_{DT}(P_D(\sO\oplus N)/N^{\infty})_{0}\\
&=&Z^{\prime}_{DT}(X_1/D, \prod_{i=1}^{r}
\tau_{q_i}(\gamma_{i}))_{\beta}
\end{eqnarray*}

Therefore it is reduced to the computation of $Z^{\prime}_{DT}(X_1,
\prod_{i\in A} \tau_{q_{i}} (\gamma_{i}))_{\beta}$. Here $X_1$
 is a compactification of the line bundle $\sO(dE)$ on a surface
 which is $P^2$ blown up at one point. For this toric case the partition
 does not vanish and the third MNOP conjecture could be checked to hold
 \cite{MOP}. So for the original three-fold $X$ which is a $\Po$ scroll
 over $BL_p(K3)$ the third MNOP conjecture is also true.

\end{document}